\documentclass[11pt,a4paper]{article}

\usepackage{inputenc}
\usepackage{amsmath}
\usepackage{bm}
\usepackage{bbold}
\usepackage{amsthm}

\usepackage{hyperref}

\title{Unbiased Estimates for Gradients\\ of Stochastic Network Performance Measures\thanks{Acta Applicandae Mathematicae, 1993. Vol.~33, pp.~21-43}
}

\author{Nikolai Krivulin\thanks{Faculty of Mathematics and Mechanics, St.~Petersburg State University, Bibliotechnaya Sq.2, Petrodvorets, St.~Petersburg, 198904 Russia}}

\date{}

\newtheorem{theorem}{Theorem}
\newtheorem{lemma}[theorem]{Lemma}
\newtheorem{corollary}[theorem]{Corollary}
\newtheorem{axiom}{Axiom}
\newtheorem{definition}{Definition}
\newtheorem{example}{Example}

\begin{document}

\maketitle

\begin{abstract}
Three classes of stochastic networks and their performance measures are 
considered. These performance measures are defined as the expected value of
some random variables and cannot normally be obtained analytically as
functions of network parameters in a closed form. We give similar
representations for the random variables to provide a useful way of analytical
study of these functions and their gradients. The representations are used to
obtain sufficient conditions for the gradient estimates to be unbiased. The
conditions are rather general and usually met in simulation study of the
stochastic networks. Applications of the results are discussed and some
practical algorithms of calculating unbiased estimates of the gradients are
also presented.
\\

\end{abstract}

\section{Introduction}

Stochastic network models are widely used in modern engineering, management,
biology, etc., to investigate real systems. These models are often so
complicated that they can hardly be studied with the help of the analytical
methods only. A more fruitful way is to use computer simulation to analyze the
networks \cite{Erma75,HoYC87,Suri89}. By performing simulation experiments,
one may get a large amount of information about the network behaviour.

Usually, the main aim of the analysis is to improve a network performance. In
order to optimize a performance criterion with respect to network parameters,
one needs to evaluate it. Simulation provides estimating the criterion as well
as its sensitivity (or its gradient, when the parameters are continuous) in a
rather simple way. It is generally not difficult to obtain the estimates
provided that there exists a simulation model, although each simulation
experiment may be very time consuming.

There are many stochastic optimization procedures which use the data obtained
by simulation (see \cite{Erma75} and also a short survey in \cite{Glyn86}). In
many cases, the procedures that exploit gradient estimates are preferred to
those using estimates of the objective function only. Specifically, the
stochastic algorithms which apply unbiased estimates of gradient are often
highly efficient. As an example, one can compare the Robbins--Monro procedure
with the Kiefer--Wolfowitz one. It is well known \cite{Glyn86} that the first
procedure based on the unbiased estimates of gradient, converges to the
solution faster than the second one which approximates the gradient by the
finite differences.

In this paper, we analyze the problem of an unbiased estimation of the
gradient of stochastic network performance measures. The paper is based on the
results obtained in \cite{Kriv90a,Kriv90b}. In Section~\ref{sec1}, we describe
three classes of stochastic networks and give some examples of the networks
and their related optimization problems. We show that the sample performance
functions of the networks of all three classes may be represented in a similar
way. In fact, these functions are expressed through those given by using the
operations of maximum, minimum and addition.

Section~\ref{sec2} includes technical results which provide a general
representation for the sample performance functions of the networks.

In Section~\ref{sec3} we briefly discuss three methods of estimating
gradients, based on simulation data.

The main results are presented in Section~\ref{sec4}. First, we introduce a
set of functions for which one may obtain unbiased estimates of their
gradients. We prove some technical lemmae to state the properties of the set.
In conclusion, we give the conditions that prove the gradient estimates to be
unbiased. These conditions are rather general and usually fulfilled in
simulation studies of the stochastic networks.

In Section~\ref{sec5} some algorithms of calculating the gradient estimates are
described.

\section{Stochastic Networks and Related Optimization Problems} \label{sec1}

In this section we present three classes of stochastic networks and discuss
their related optimization problems. A performance criterion of the network is
normally defined as the expected value of a random variable,
$  f(\theta,\omega)  $, i.e.,
$$
F(\theta) = E_{\omega}[f(\theta,\omega)] = E[f(\theta,\omega)],
$$
where $  \theta \in \Theta \subset R^{n}  $ is a set of decision parameters
and $ \omega  $ is a random vector representing the randomness of network
behaviour. As a function of the parameters, $  f(\theta,\omega)  $ is
usually called a sample performance function.

The problem is to optimize the performance measure $  F(\theta)  $ with
respect to $  \theta \in \Theta $. In practical problems, it is often very
hard to evaluate the expectation analytically in closed form, even if there is
an analytical formula available for $  f(\theta,\omega) $. However, it is
normally not difficult to obtain the value of $  f(\theta,\omega)  $ for any
fixed $  \theta \in  \Theta  $ and any realization of $  \omega  $ by
using simulation. In that case, one can use the Monte Carlo approach to
estimate the objective function $  F(\theta)  $ or its gradient.

The main purpose of this section is to show that for many optimization
problems, the sample performance function $  f(\theta,\omega)  $ may be
represented in similar algebraic forms. In other words,
$  f(\theta,\omega)  $ is expressed in terms of some given random variables
by means of the operations $  \max $, $  \min $, and $  + $. This
representation offers the potential for analytical study of the estimates of
performance measure gradients. It also provides a theoretical background for
efficient algorithms of calculating the estimates.

\subsection{Activity Network}
We begin with stochastic activity network models widely used in corporate
management in the scheduling of large projects. Consider a project consisting
of some activities (or jobs) which must be done to complete it. Each activity
is presumed to require a random amount of time for performing it. It is not
permitted to begin each activity until some preliminary activities have been
performed. One is normally interested in reducing the expected completion time
of the whole project.

In order to describe the project as a network, we define an oriented graph
$  ({\bf N, A})$, where $  {\bf N}  $ is the set of nodes and
$ {\bf A}  $ is the set of arcs. Each node $  i \in {\bf N}  $ represents
the corresponding activity of the project. For some $  i, j \in {\bf N}$,
the arc $  (i,j) $ belongs to $  {\bf A}  $ if and only if the $i$th
activity must directly precede the $j$th one.

To simplify further formulae we define the set of the father nodes as
$  {\bf N}_{F}(i) = \{ j \in {\bf N} | (j,i) \in {\bf A} \} $, and the set
of the daughter nodes as
$  {\bf N}_{D}(i) = \{ j \in {\bf N} | (i,j) \in {\bf A} \}  $ for every
$  i \in {\bf N} $. In addition, we introduce the set of starting nodes
$  {\bf N}_{S} = \{ i \in {\bf N} | {\bf N}_{F} (i) = \emptyset \}  $ and
the set of the end nodes
$  {\bf N}_{E} = \{ i \in {\bf N} | {\bf N}_{D} (i) = \emptyset \}  $ of the
graph.

Now we have to define the duration of the activities, so that the network
would be completely described. Denote the duration of the $i$th activity by
$  \tau_{i} $, $  i \in {\bf N}$. We assume $  \tau_{i}  $ to be a
positive random variable, such that $  \tau_{i} = \tau_{i}(\theta,\omega) $,
where $  \theta \in \Theta  $ is a set of decision parameters and
$  \omega  $ is a random vector which represents the random effects involved
in realizing the project. The set
$  {\bf T} = \{ \tau_{i} | i \in {\bf N} \}  $ is presumed to be given.

The sample completion time of the $i$th activity may be expressed in the form
\begin{equation} \label{equ1}
t_{i}(\theta,\omega) = \left\{ \begin{array}{ll}
              \max_{j \in {\bf N}_{F}(i)} t_{j}(\theta,\omega) +
              \tau_{i}(\theta,\omega), &
                                \mbox{if $  i \not\in {\bf N}_{S} $}, \\
              \tau_{i}(\theta,\omega), &
                                \mbox{if $  i \in {\bf N}_{S} $}.
              \end{array}
\right.
\end{equation}
For the sample completion time of the whole project, we have
$$
t(\theta,\omega) = \max_{i \in {\bf N}_{E}} t_{i}(\theta,\omega).
$$
In that case, the expected completion time is
$$
T(\theta) = E[t(\theta,\omega)],
$$
and the problem is to minimize $  T(\theta)  $ with respect to
$  \theta \in \Theta $.

It is easy to see from (\ref{equ1}) that one can represent $  t  $ as a
function of $  \tau \in {\bf T}  $ by using the operations $  \max  $ and
$ +$. To illustrate this, consider the simple network depicted in
Figure~\ref{fig1}.
\begin{figure}[hhh]
\begin{center}
\begin{picture}(110,60)
\put(2,27){\circle{10}}
\put(1,26){1}
\put(27,52){\circle{10}}
\put(26,51){2}
\put(27,2){\circle{10}}
\put(26,1){3}
\put(77,52){\circle{10}}
\put(76,51){4}
\put(77,2){\circle{10}}
\put(76,1){5}
\put(102,27){\circle{10}}
\put(101,26){6}

\put(6,31){\vector(1,1){17}}
\put(6,23){\vector(1,-1){17}}
\put(32,52){\vector(1,0){40}}
\put(32,2){\vector(1,0){40}}
\put(31,6){\vector(1,1){42}}
\put(81,48){\vector(1,-1){17}}
\put(81,6){\vector(1,1){17}}

\end{picture}
\end{center}

\caption{An activity network.}\label{fig1}
\end{figure}
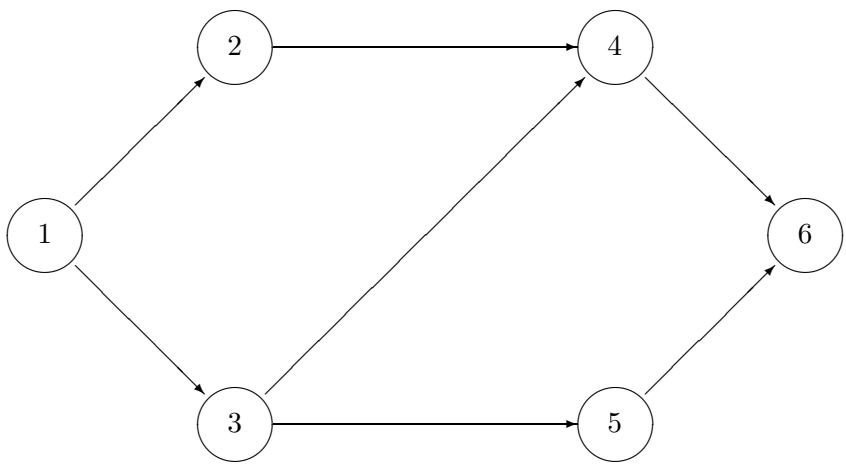

For this network, applying (\ref{equ1}) successively, we may write the sample
completion time as
$$
t = \tau_{1} + \max\{\max\{\tau_{2},\tau_{3}\}+\tau_{4},\tau_{3}+\tau_{5}\}
  + \tau_{6}.
$$
We will exploit the possibility of $  t  $ being expressed in such a form
in the discussion below.

We conclude this example with the remark about the main difficulty of the
activity network optimization problem. It is easy to understand that in the
case of general random variables $  \tau \in {\bf T} $, it is usually very
difficult or even impossible to obtain the expected completion time
analytically, even if the network is as simple as that in Figure~\ref{fig1}.
To apply an optimization procedure in this situation, one normally estimates
this function or its gradient by using the Monte Carlo approach. Notice,
however, that the simulation models of such networks are generally rather
simple.

\subsection{Reliability Network}
Another class of stochastic network models arises from the reliability
investigation of complex interconnected systems in engineering, military
research, biology etc. Consider a system of elements having bounded random
lifetimes. Each element keeps in order until either this element has failed or
all those directly supplying it have lost their working conditions. The whole
system is presumed to be in order if at least one of the main elements that
are supplied by some others but do not supply any element, keeps working. A
usual problem in analyzing the system is to maximize its expected lifetime.

Let $  ({\bf N, A})  $ be the directed graph describing the relations
between the system elements. In the graph, the set of nodes $  {\bf N}  $
corresponds to the set of system elements. If for some $  i,j \in {\bf N} $,
the $i$th element directly supplies the $j$th one, then
$  (i,j) \in {\bf A} $. For the graph, we retain the notations
$  {\bf N}_{F}(i), {\bf N}_{D}(i), {\bf N}_{S} $, and $  {\bf N}_{E}  $
introduced above.

For each element $  i \in {\bf N} $, we define the lifetime as the random
variable $  \tau_{i} (\theta,\omega)  $  which depends on the set of
decision parameters $  \theta \in \Theta $. Assume the set
$  {\bf T} = \{\tau_{i}\}  $ to be given. Now, we may represent the time for
the $i$th element to be in order as
\begin{equation}\label{equ2}
t_{i}(\theta,\omega) = \left\{ \begin{array}{ll}
              \min\{\max_{j \in {\bf N}_{F}(i)} t_{j}(\theta,\omega),
                                                 \tau_{i}(\theta,\omega) \}, &
                                \mbox{if $  i \not\in {\bf N}_{S} $}, \\
              \tau_{i}(\theta,\omega), & \mbox{if $  i \in {\bf N}_{S} $}.
              \end{array}
\right.
\end{equation}

The sample and expected lifetimes of the whole system may be written as
$$
t(\theta,\omega) = \max_{i \in {\bf N}_{E}} t_{i} (\theta,\omega) \qquad
\mbox{and} \qquad
T(\theta) = E[t(\theta,\omega)],
$$
respectively.

To illustrate this reliability network model, consider that depicted in
Figure~\ref{fig2} (\cite{Erma75}).
\begin{figure}[hhh]

\begin{center}
\begin{picture}(100,50)
\put(0,20){\framebox(10,10){1}}
\put(30,40){\framebox(10,10){2}}
\put(30,20){\framebox(10,10){3}}
\put(30,0){\framebox(10,10){4}}
\put(60,20){\framebox(10,10){5}}
\put(60,0){\framebox(10,10){6}}
\put(90,20){\framebox(10,10){7}}

\put(10,25){\vector(1,1){20}}
\put(10,25){\vector(1,0){20}}
\put(10,25){\vector(1,-1){20}}
\put(40,45){\vector(1,-1){20}}
\put(40,25){\vector(1,0){20}}
\put(40,5){\vector(1,0){20}}
\put(70,25){\vector(1,0){20}}
\put(70,5){\vector(1,1){20}}

\end{picture}
\end{center}

\caption{A reliability network.}\label{fig2}
\end{figure}
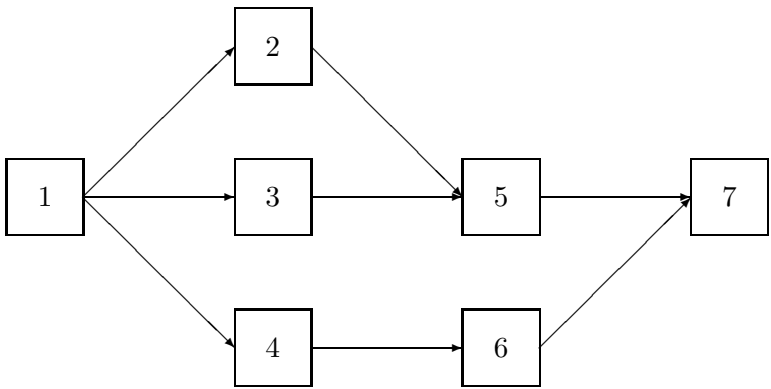

For the sample lifetime of the system, we have from (\ref{equ2})
$$
t = \min\{\tau_{1},\max\{\min\{\tau_{4},\tau_{6}\},
    \min\{\max\{\tau_{2},\tau_{3}\},\tau_{5}\}\},\tau_{7}\}.
$$

We can see that the sample lifetime of such a network has one important
property: it may be represented as a function of all $  \tau \in {\bf T}  $
by using only the operations $  \max  $ and $  \min $. Note that the
difficulties in solving the problem of expected lifetime maximization are the
same as in activity network optimization.

\subsection{Queueing Network}
Queueing networks provide a very rich class of stochastic network models.
These models play the key role in simulation study of computer systems,
communication networks, production lines, flexible manufacturing systems, etc.
In this part of the section, a general class of queueing networks is
considered and several performance measures of the network are defined.

The network which we describe consists of $  L  $ single--server nodes.
There are a server and a buffer with infinite capacity in each node $  i $,
$  i=1,\ldots,L$. Once a customer arrives into node $  i $, he occupies the
server if it is free. The server keeps busy a random amount of time until the
service of the customer has been completed. Upon the completion of its service
at node $  i $, the customer goes to node $  j $, chosen according to some
routing procedure described below. We suppose that the customer arrives
immediately into node $  j $.

The customer may find the server of node $  i  $ being busy. In that case,
he joins the queue at the node and is placed into the buffer. The discipline
in which queued customers are called forward for service is first--come,
first--served. We assume that at the initial time, all servers of the network
are free and there are $  n \; (0 \leq n \leq \infty)  $ customers in the
buffer at node $  i $, $  i=1,\ldots,L$. The customers are presumed to be of
a single class.

For the network, define the set of random variables
$  {\bf T} = \{\tau_{ij}\} $, where $  \tau_{ij}(\theta,\omega)  $ is the
service time of the customer that is the $j$th to initiate a service at node
$  i $. These random variables depend on the set of decision parameters
$  \theta \in \Theta  $ and a random vector $  \omega $, and they are
presumed to be given data.

Now, we discuss a routing mechanism of the network. A general routing
procedure may be defined by means of the set of random variables,
$  \Sigma = \{\sigma_{ij}\} $, where $  \sigma_{ij}(\omega)  $ represents
the next node to be visited by the customer who is the $j$th to depart from
node $  i $. Let $  s_{ij}  $ be a realization of
$  \sigma_{ij}(\omega)  $ for a fixed $  \omega $,
$  s_{ij} \in \{1,\ldots,L\} $. The matrix
$$
S = \left(
          \begin{array}{ccc}
            s_{11} & s_{12} & \ldots \\
            s_{21} & s_{22} & \ldots \\
            \ldots & \ldots & \ldots \\
            s_{L1} & s_{L2} & \ldots
          \end{array}
    \right)
$$
is referred to as a routing table of the network. In addition, denote the set
of all possible routing tables of the network by $  {\bf S} $.

Firstly, we consider a special case of the network to demonstrate the relation
between the queueing networks and those we have described above. Let us fix a
routing table $  S \in {\bf S}  $ and consider the network with the
deterministic routing procedure defined by $  S $. One can state a lot of
optimization problems of the network. The problems may differ in the
performance criteria to be optimized. In order to produce useful
representations of sample performance functions of the network, we introduce
the following notations. For every node $  i $, $  i = 1,\ldots,L $, let
\vspace{1ex}

$\alpha_{ij}  $ be the time of the $j$th arrival into the node,

$\beta_{ij}  $ be the time of the $j$th initiation of a service,

$\delta_{ij}  $ be the time of the $j$th departure from the node.
\vspace{1ex}

Obviously, for each node $  i  $ we may analytically represent the
relationship between these variables as follows
\begin{equation}\label{equ3}
\left\{
\begin{array}{lcl}
\delta_{ij} & = & \beta_{ij} + \tau_{ij}, \\
\beta_{ij}  & = & \max \{\alpha_{ij},\delta_{ij-1}\}, \quad j = 1,2,\ldots,
                                    \qquad \delta_{i0} \equiv 0.
\end{array}
\right.
\end{equation}
It should be noted that in the above identities each $  \alpha_{ij}  $
coincides with some $  \delta_{mk}  $ because the transition of customers
from one node to another is immediate.

One of the performance measures of the network is the expected value of the
$M$th service completion time at node $  K $
$$
D(\theta) = E[\delta_{KM}(\theta,\omega)]
$$
that we wish to minimize with respect to $  \theta \in \Theta  $ for given
$  K  $ and $  M $. Some other performance criteria one usually choose to
optimize will be defined below.

The following fact is of great importance. For the network with a
deterministic routing procedure, the sample completion time
$  \delta_{ij}  $ may be represented as a function of
$  \tau \in {\bf T}  $ using the operations $  \max $, $  \min $, and
$  +  $ for every $  i, j  $ provided that this service occurs.

To illustrate this, consider an example of network with three nodes and the
routing table
$$
S = \left(
      \begin{array}{ccccc}
        2 & 1 & 1 & 3 & \ldots \\
        1 & 3 & 1 & 1 & \ldots \\
        2 & 3 & 1 & 2 & \ldots \\
      \end{array}
    \right).
$$

Assume $  n_{1} = n_{2} = n_{3} = 1  $ and choose a node of the network, say
node 2. Since there is a customer at node 2 at the initial time, using
(\ref{equ3}) we may write
$$
\begin{array}{lll}
\alpha_{21} = 0, & \beta_{21} = 0, & \delta_{21} = \tau_{21}.
\end{array}
$$

It is easy to see from $  S  $ that the customer who is the second to arrive
into node 2 will be one of those first serviced at both nodes 1 and 3. It will
be just that customer who completes his service earlier. In that case we have
$$
\begin{array}{lll}
  \alpha_{22} =\min \{\tau_{11}, \tau_{31} \}, &
  \beta_{22} = \max \{\alpha_{22}, \delta_{21} \}, &
  \delta_{22} = \beta_{22} + \tau_{22}.
\end{array}
$$

Based on $  S $, we also deduce that the other customer from these two will
be the third to arrive into node 2. Therefore, we may write
$$
\begin{array}{lll}
  \alpha_{23} =\max \{\tau_{11}, \tau_{31} \}, &
  \beta_{23} = \max \{\alpha_{23}, \delta_{22} \}, &
  \delta_{23} = \beta_{23} + \tau_{23}.
\end{array}
$$

Finally, for $  \delta_{23}  $ we get
$$
\delta_{23} = \max \{\max\{\tau_{11},\tau_{31}\},
          \max\{\min\{\tau_{11},\tau_{31}\},\tau_{21}\}+\tau_{22}\}+\tau_{23}.
$$

As we can see, there is the representation of the sample performance function
of the network like that we have pointed out in the previous examples. Notice
that it may be very difficult to obtain such representations in practice.
However, it is essential that the representation exists. We will give reason
for this fact in the next section.

Now we define some traditional performance measures of the queueing networks.
These definitions are also extended to the general network with the stochastic
routing procedures.

Suppose that we observe the network until the $M$th service completion at node
$K$, for given $  K, M $. As sample performance functions for node $  K  $
in the observation period, we consider the following:
\vspace{1ex}

$ t(\theta,\omega) = \frac{1}{M} \sum_{j=1}^{M} \; (\delta_{Kj}(\theta,\omega)
- \alpha_{Kj}(\theta,\omega)), \;$ the average total time per customer,
\vspace{1ex}

$ w(\theta,\omega) = \frac{1}{M} \sum_{j=1}^{M} \; (\beta_{Kj}(\theta,\omega)
- \alpha_{Kj}(\theta,\omega)), \;$ the average waiting time per customer,
\vspace{1ex}

$ u(\theta,\omega) = \sum_{j=1}^{M} \; \tau_{Kj}(\theta,\omega)/
\delta_{KM}(\theta,\omega), \;$ the average utilization per unit time,
\vspace{1ex}

$ c(\theta,\omega) = \sum_{j=1}^{M} \; (\delta_{Kj}(\theta,\omega) -
\alpha_{Kj}(\theta,\omega)) / \delta_{KM}(\theta,\omega), \;$
the average number of customers per unit time,
\vspace{1ex}

$ q(\theta,\omega) = \sum_{j=1}^{M} (\beta_{Kj} (\theta,\omega) -
\alpha_{Kj}(\theta,\omega)) / \delta_{KM}(\theta,\omega), \;$
the average queue length per unit time.
\vspace{1ex}

Denote the expected values of these sample functions with respect to
$  \omega  $ by $  T(\theta) $, $  W(\theta) $, $  U(\theta) $,
$  C(\theta) $, and $  Q(\theta) $, respectively. They are the performance
measures that one usually chooses to optimize the network.

Note that the sample performance functions are described in terms of elements
of the set $  \{\alpha\} $, $  \{\beta\} $, $  \{\delta\} $, and
$  \{\tau \} $. It will be shown in the next section that each of the random
variables $  \alpha $, $  \beta $, and $  \delta  $ is expressed by a
function of $  \tau \in {\bf T} $, using the operations $  \max $,
$  \min $, and $  + $. This circumstance is of great importance and it will
be necessary to study analytical properties of the performance measures and
their gradient estimates in Section~\ref{sec4}.

For the general routing procedure defined by the set $  \Sigma $, any of the
above sample performance functions may be represented as
\begin{equation}\label{equ4}
f(\theta,\omega)
        = \sum_{S \in {\bf S}} \; 1_{[\Sigma(\omega)=S]} f_{S}(\theta,\omega).
\end{equation}
Here $  1_{[\Sigma(\omega)=S]}  $ is the indicator of the event
$  \{\Sigma(\omega)=S\} $, and $  f_{S}(\theta,\omega)  $ is the sample
performance function that coincides with $  f(\theta,\omega)  $ provided the
routing procedure is deterministic and defined by the routing table $  S $.

It should be noted in conclusion, that the sample performance functions of the
queueing network are more difficult to take their expectation analytically
than those of the networks we have considered above.

\section{Algebraic Representations for the Networks}\label{sec2}

We have seen that the functions of network performance possess some algebraic
properties. The point is that they may be expressed as a function of given
random variables by means of the operations $  \max $, $  \min $, and
$  + $. For the activity networks and reliability, this follows directly from
recursive equations (\ref{equ1}) and (\ref{equ2}) and does not require any
special proofs. The possibility of such representations in describing of the
sample performance functions of the queueing network is not obvious. In this
section, we give the theorem that states the existence of the representation
in the case of the network with a deterministic routing procedure. Two
technical lemmae are also included in the section.

In order to simplify further formulae, we introduce the notations $  \vee  $
for maximum and $  \wedge  $ for minimum. In addition, we will use the sign
$  \bigvee \; (\bigwedge)  $ to represent an iterated maximum (minimum),
i.e.,
$$
\bigvee_{i=1}^{n}   x_{i} = x_{1} \vee \ldots \vee x_{n} \quad
\left(\bigwedge_{i=1}^{n}   x_{i} = x_{1} \wedge \ldots \wedge x_{n}\right).
$$

Let $  {\bf X}  $ be a set supplied with the operations $  + $, $  \vee $,
and $  \wedge $. Without loss of generality, we may consider $  {\bf X}  $
to be a set of real numbers. It is easy to extend the result of this section
to various sets of real--valued functions and random variables. We assume that
the traditional algebraic axioms are fulfilled in $  {\bf X} $. In
particular, we will use the following axioms.

\begin{axiom}\label{axi1}
Distributivity of maximum over minimum.
$$
\forall x,y,z \in {\bf X}, \quad
(x \wedge y)\vee z = (x \vee z)\wedge(y \vee z).
$$
\end{axiom}

\begin{axiom}\label{axi2}
Distributivity of minimum over maximum.
$$
\forall x,y,z \in {\bf X}, \quad
(x \vee y)\wedge z = (x \wedge z)\vee(y \wedge z).
$$
\end{axiom}

\begin{axiom}\label{axi3}
Distributivity of sum over maximum and minimum.
$$
\forall x,y,z \in {\bf X}, \quad
(x \vee y)+z = (x+z)\vee(y+z), \quad (x \wedge y)+z = (x+z)\wedge(y+z).
$$
\end{axiom}

The proof of the representation theorem for the queueing networks is based on
the next result. Let $  X = \{x_{1},\ldots,x_{n}\}  $ be a finite set of
real numbers. Suppose that we arrange its elements in order of increase, and
denote the $k$th smallest element by $  x_{(k)} $. If there are elements of
an equal value, we count them repeatedly in an arbitrary order.

\begin{lemma}\label{lem1}
For each $  k = 1,\ldots,n $, the value of $  x_{(k)}  $ is
given by
\begin{equation}\label{equ5}
x_{(k)} = \bigwedge_{I \in \Im_{k}} \; \bigvee_{i \in I} \: x_{i},
\end{equation}
where $  \Im_{k}  $ is the set of all $k$ subsets of the set
$  N = \{1,\ldots,n\} $.
\end{lemma}

\begin{proof}
Denote the set of indices of the first $  k   $ smallest
elements by $  I^{*} $. It is clear that
$  x_{(k)} = \bigvee_{i \in I^{*}} x_{i} $. Consider an arbitrary subset
$  I \in \Im_{k} $. Obviously, if $  I \neq I^{*} $, there is at least one
index $  j \in I  $ such that $  x_{j} \geq x_{(k)} $. Therefore, we have
$  x_{(k)} \leq \bigvee_{i \in I} x_{i} $. It remains to take minimum over
all $  I \in \Im_{k}  $ in the last inequality so as to get (\ref{equ5}).
\end{proof}

\begin{theorem}
Let $  S \in {\bf S}  $ be a fixed routing table. For the network with the
deterministic routing procedure defined by $  S $, every $  \alpha_{ij} $,
$  \beta_{ij} $, and $  \delta_{ij}  $
$(i = 1,\ldots,L; \; j=1,2,\ldots) $ is represented as a function of
$  \tau \in {\bf T} $, using the operations $  \max $, $   \min  $ and
$  + $, provided that its associated service occurs.
\end{theorem}

\begin{proof}[Sketch of the proof.]
Consider a network of $  L  $ nodes with a
routing table $  S $. The recursive equations of the network are written as
\begin{equation}\label{equ6}
\begin{array}{lcl}
\delta_{ij} & = & \beta_{ij} + \tau_{ij}, \\
\beta_{ij}  & = & \alpha_{ij} \vee \delta_{ij-1}, \quad j = 1,2,\ldots,
                                    \quad \delta_{i0} \equiv 0,
\end{array}
\end{equation}
for each node $  i $, $  i = 1,\ldots,L $. The main idea of our proof is to
reduce these equations to the one that expresses the departure time
$  \delta_{ij}  $ through both the service time $  \tau_{ij}  $ and some
other departure times $  \delta_{km}  $
$(k \in \{1,\ldots,L\} $, $  m \in \{1,2,\ldots\})  $ only.

Let us first examine $  \alpha_{ij} $, i.e. the arrival time of the customer
who is the $j$th to come into node $  i $. It is plain that this customer is
one of those who are to leave some other nodes to go to node $  i $.
Therefore, $  \alpha_{ij}  $ coincides with one of the departure times
$  \delta_{km}  $ such that $  s_{km} = i $.

Consider the customers who have to go to node $  i  $ from some other, say
node $  k $. Clearly, we may restrict ourselves to the first $  j  $
customers because this is enough to provide the $j$th customer to come into
node $  i $. Denote the set of the times at which customers depart from node
$  k  $ by $  \Delta_{k} (i,j) $, $  k \neq i $.

It may happen that $  s_{im} = i  $ for some $  m $. This means that a
customer who is the $m$th serviced at node $  i  $ should join the queue
of the same node again. In this case, we have to take account of such
customers being served before the $j$th one only. The corresponding set of the
departure times is
$  \Delta_{i}(i,j) = \{\delta_{im} | s_{im} = i \; \mbox{and} \; m < j \} $.

Let $  \Delta(i,j) = \cup_{k=1}^{L} \Delta_{k}(i,j) $. Note that if there
are $  n_{i} > 0  $ customers in the buffer of node $  i  $ at the initial
time, then $  \alpha_{i1} = \cdots = \alpha_{in_{i}} = 0 $. It is plain that
for $  j > n_{i}  $ the arrival time $  \alpha_{ij}  $ coincides with the
$(j-n_{i})$th smallest element of $  \Delta(i,j) $. It follows from
Lemma~\ref{lem1} that $  \alpha_{ij}  $ is represented as a function of
elements of the set $  \Delta(i,j)  $ by using the operations $  \max $ and
$  \min $. In addition, using (\ref{equ6}) we may get such a representation
for $  \beta_{ij}  $ as a function of departure times
$  \delta \in \Delta(i,j) \cup \{\delta_{ij-1}\} $.

Finally, it follows from (\ref{equ6}) that there exists a representation of
$  \delta_{ij}  $ as the function of both $  \tau_{ij}  $ and the elements
of the set $  \Delta(i,j) \cup \{\delta_{ij-1} \}  $ which is a
superposition of the operations $  \max $, $  \min $, and $  + $. One can
resolve this equation and get $  \delta_{ij}  $ as a function of the service
times $  \tau \in {\bf T}  $ only. This produces the representation that the
theorem requires.
\end{proof}

We conclude this section with the following technical lemma that offers a
general form of the representation.

\begin{lemma}\label{lem3}
Let $  \varphi(x_{1},\ldots,x_{p})  $ be a function of the variables
$  x_{1},\dots,x_{p}  $ taking their values in $  {\bf X} $,
$  \varphi  $ is defined as a composition of the operations $  \vee $,
$  \wedge $, and $  + $. Then $  \varphi  $ can be represented as
$$
\varphi (x_{1},\ldots,x_{p} ) = \bigvee_{i \in I} \;
\bigwedge_{j \in J_{i}} \; \sum_{k=1}^{p} \alpha_{ij}^{k}   x_{k},
$$
where $  I  $  and $  J_{i}  $ for all $  i \in I  $ are finite
sets of indices, and all $  \alpha^{k}_{ij}  $ are integers.
\end{lemma}

\begin{proof}
Without loss of generality we suppose that there is no more than
one entry of each variable $  x_{1},\ldots,x_{p}  $ into the expression.
If some variable has two or more entries, we introduce additional ones so that
the above presupposition would be fulfilled. Let us prove the lemma by
induction on the number of variables.

For $  p=1 $, the statement of the lemma is obvious. If $  p=2 $, there are
three possibilities
$$
x_{1} \vee x_{2}, \quad
x_{1} \wedge x_{2} \quad \mbox{and} \quad x_{1} + x_{2},
$$
and it is clear that the statement is also true.

Assume that the statement of the lemma is true up to some value $  p-1 $.
Consider an expression $  \varphi  $ of $  p  $ variables. Clearly, there
is an operation in the expression that should be performed after the other
ones. Denote this operation by the asterisk $ *.$ In this case, we have
$  \varphi = \varphi_{1} * \varphi_{2} $, where $  \varphi_{1}  $ and
$  \varphi_{2}  $ are expressions such that each of them cannot include all
the variables $  x_{1},\ldots,x_{p} $. By the assumption, the statement of
the lemma holds for both $  \varphi_{1}  $ and $  \varphi_{2} $. Now, we
have three possibilities for the operation $  * $.
\begin{enumerate}
\item  $  \vee $. This is obvious.
\item  $  \wedge $. It is sufficient to apply Axiom~\ref{axi1}.
\item  $  + $. To obtain the representation in this case, one has to apply
successively Axioms~\ref{axi1}, \ref{axi2} and \ref{axi3}.
\end{enumerate}
Consequently, the statement of the lemma is true for
$  \varphi = \varphi_{1} * \varphi_{2} $.
\end{proof}

\section{Estimates of Gradient}\label{sec3}

To optimize the network performance measure
$  F(\theta)=E[f(\theta,\omega)] $, one often needs information about the
gradients $  \partial F(\theta)/\partial\theta $. In the absence of
analytical formulae for the gradient, Monte Carlo experiments may be performed
to estimate its values. There are three general methods of estimating
$  \partial F(\theta)/\partial\theta  $ based on data obtained by simulation
\cite{Erma75,Suri89,CaoX85}. In the first two methods, the gradient is
approximated by the finite differences and then estimated by using the Monte
Carlo approach. To illustrate these two methods, assume $  \theta  $ to be a
scalar and consider the following estimates:

The crude Monte Carlo (CMC) estimate:
$$
G_{CMC} = \frac{1}{N\Delta\theta} \sum_{i=1}^{N}
\left(f(\theta+\Delta\theta,\omega_{i})-f(\theta,\omega_{N+i})\right);
$$

the common random number (CRN) estimate:
$$
G_{CRN} = \frac{1}{N\Delta\theta} \sum_{i=1}^{N}
\left(f(\theta+\Delta\theta,\omega_{i})-f(\theta,\omega_{i})\right),
$$
where $  \omega_{i} $, $  i=1,\ldots,2N  $ are independent realizations of
the random vector $  \omega $. The second estimate differs from the first in
one respect: in the CRN estimate the random variables $  \omega_{i}  $ are
the same for both $  \theta+\Delta\theta  $ and $  \theta $, whereas in the
CMC estimate they are different. Note that each of them requires
$  2 \times N  $ simulation runs ($ N  $ at the original value
$  \theta  $ and $  N  $ at $  \theta+\Delta\theta$). Clearly, in the
case of the vector $  \theta \in R^{n} $, one must perform
$  (n+1)\times N  $ simulation experiments to get each estimate. In
\cite{Erma75} has shown that the finite difference estimates have the mean
square error (MSE) which is of order $  O(N^{-1/3})  $ for $  G_{CMC}  $
and $  O(N^{-1/2})  $ for $  G_{CRN}$.

We may somewhat improve the MSE properties of the estimate by using more
sophisticated finite difference formulae. However, the estimates become very
expensive in terms of computation time because they require a large number of
additional simulation experiments. For example, the following symmetric
difference estimate
$$
G_{CRN}^{SD} = \frac{1}{2N\Delta\theta} \sum_{i=1}^{N}
                 \left(f(\theta+\Delta\theta,\omega_{i})
             - f(\theta-\Delta\theta,\omega_{i})\right)
$$
requires $  2 \times n \times N  $ simulation runs, when
$  \theta \in R^{n}$).

An estimate of the third method can be written in the form
\begin{equation}\label{equ7}
G =  \frac{1}{N}
\sum_{i=1}^{N} \frac{\partial}{\partial\theta} f(\theta,\omega_{i}),
\end{equation}
provided that the gradient of the sample performance function (sample
gradient) exists. It should be noted that although we may obtain values of the
sample performance function by simulation, it can be rather difficult to
evaluate its gradient.

Recently, a new technique called infinitesimal perturbation analysis (IPA) has
been developed \cite{HoYC87} as an efficient method of obtaining gradient
information. The IPA method yields the exact values of the sample
gradient $  \partial f(\theta,\omega)/\partial\theta  $ by performing one
simulation run. The method is based on the analysis of the dynamics of the
network and closely connected with the simulation technique. Therefore, one
can easily combine an IPA procedure for calculating the sample gradient with a
suitable algorithm of network simulation. Such a procedure provides all the
partial derivatives of the sample gradient simultaneously during one
simulation run. Furthermore, it needs an additional computation cost which is
usually very small compared with that required for the simulation run alone.

The key question concerning the IPA method is whether it produces an unbiased
estimate of the performance measure gradient. It can easily be shown that if
$  \partial f(\theta,\omega)/\partial\theta  $ is an unbiased estimator of
$  \partial F(\theta)/\partial\theta  $ then estimate (\ref{equ7}) has MSE
which is of order $  O(N^{-1}) $. In short, in the case of unbiasedness, this
is a very efficient estimate which provides considerable savings in
computation.

In the next section, using the algebraic representation of Section~\ref{sec2},
we will examine properties of the network performance functions so as to
derive the conditions for estimate (\ref{equ7}) to be unbiased.

\section{A Theoretical Background of Unbiased Estimation}\label{sec4}

It is easy to understand that a sufficient condition for the estimate
(\ref{equ7}) of the gradient
$  \partial E[f(\theta,\omega)]/\partial\theta  $ at some
$  \theta \in \Theta  $ to be unbiased is
\begin{equation}\label{equ8}
\frac{\partial}{\partial\theta} E[f(\theta,\omega)] =
E\left[\frac{\partial}{\partial\theta} f(\theta,\omega)\right].
\end{equation}

Cao showed in \cite{CaoX85} that (\ref{equ8}) holds in the case of
$  f(\theta,\omega)  $ being uniformly differentiable at $  \theta  $
w.p.~1. Note that such a differentiability property is not easy to verify and
hard to interpret for practical systems. A useful way to prove the interchange
in (\ref{equ8}) is to apply the Lebesgue dominated convergence theorem
\cite{Loev60}. We use this theorem in the following form.

\begin{theorem}\label{the4}
Let $  (\Omega, {\cal F}, P)  $ be a probability space,
$  \Theta \subset R^{n}  $ and
$  f: \Theta \times \Omega \longrightarrow R  $ be a
${\cal F}$--measurable function for any $  \theta \in \Theta  $ and such
that the following conditions hold:
\begin{list}{}{}
\item[{\rm (i)}] for every $  \theta \in \Theta $, there exists
$  \partial f(\theta,\omega)/\partial\theta  $ at $  \theta  $ w.p.~1,
\item[{\rm (ii)}] for all $  \theta_{1},\theta_{2} \in \Theta $, there is a
random variable $  \lambda(\omega)  $ defined on the same probability space,
with $  E\lambda < \infty  $ and such that
\begin{equation}\label{equ9}
     |f(\theta_{1},\omega)-f(\theta_{2},\omega)| \leq \lambda(\omega)
                                       \|\theta_{1}-\theta_{2}\| \quad w.p.~1.
\end{equation}
\end{list}

Then equation (\ref{equ8}) holds on $  \Theta $.
\end{theorem}

As an important consequence, we may state that the function
$  F(\theta) = E[f(\theta,\omega)]  $ is a Lipschitz one with a constant
$  L = E\lambda  $ and continuously differentiable on $  \Theta $, provided
$  f  $ satisfies the theorem conditions.

\begin{definition}\label{def1}
A function $  f(\theta,\omega)  $ defined on the probability space
$  (\Omega, {\cal F}, P)  $ at every $  \theta \in \Theta  $ belongs to
the set $  {\cal D}_{\Theta,\Omega}$ (or simply $  {\cal D}$) if and only if
it satisfies the conditions of Theorem~\ref{the4}.
\end{definition}

\begin{example}\label{exa1}
{\rm
Random variables which arise from a simulation study of networks, can be
treated as members of a family of random variables \cite{Erma75}. There are
few families one usually applies, namely the Exponential family, the Gaussian
family, etc. Various random variables of a family may be obtained from the
standard variable by using a suitable transformation. An ordinary way to
transform random variables is based on changing the location and scale
parameters.

Let $  \xi(\omega)  $ be the standard random variable of a family. Define
$$
f(\theta,\omega) = \theta_{1}\xi(\omega) + \theta _{2},
$$
where $  \theta = (\theta_{1},\theta_{2})^{\top} \in \Theta \subset R^{2}$.
Let us check whether it holds that $  f \in {\cal D} $. Obviously, the
partial derivatives of $  f  $ with respect to $  \theta_{1}  $ and
$  \theta_{2}  $ exist for almost all $  \omega  $ and equal
$$
\frac{\partial}{\partial\theta_{1}} f(\theta,\omega) = \xi(\omega) \quad
\mbox{and} \quad \frac{\partial}{\partial\theta_{2}} f(\theta,\omega) = 1.
$$
In addition, it is easy to verify that $  f  $ satisfies the condition~(ii)
of Theorem~\ref{the4} with $  \lambda = |\xi|+1 $. If $  E|\xi| < \infty $,
as is usually the case, then the conditions of Theorem~\ref{the4} are
fulfilled for $  f  $ and we have $  f \in {\cal D} $.
}
\end{example}

The next technical lemmae give the sufficient conditions for the arithmetic
operations and the operation $  \max  $ and $  \min  $ not to break the
main properties of the functions from $  {\cal D} $.

\begin{lemma}\label{lem5}
Let f,g $  \in {\cal D}  $ and let $  \lambda_{1}  $ and
$  \lambda_{2}  $ be the random variables that provide the condition~(ii) of
Theorem~\ref{the4} for $  f  $ and $  g $, respectively. Let
$  \mu_{1}, \mu_{2}  $ and $  \nu  $ be positive random variables. Then
the following are satisfied
\begin{list}{}{}
\item[{\rm (i)}] $  f+g \in {\cal D} $;
\item[{\rm (ii)}] If $  \alpha  $ is a bounded random variable, then
$  \alpha f \in {\cal D} $;
\item[{\rm (iii)}] If $  |f| \leq \mu_{1}  $ and $  |g| \leq \mu_{2}  $
hold w.p.~1 for any $  \theta \in \Theta  $ and
$  E[\lambda_{1}\mu_{2} + \lambda_{2}\mu_{1}] < \infty $, then
$  fg \in {\cal D} $;
\item[{\rm (iv)}] If $  |f| \leq \mu_{1}  $ and $  |g| \geq \nu  $ hold
w.p.~1 for any $  \theta \in \Theta  $ and
$  E[\mu_{1}\lambda_{2}/\nu^{2} + \lambda_{1}/\nu] < \infty $, then
$  f/g \in {\cal D} $.
\end{list}
\end{lemma}

\begin{proof}
Clearly, $  f+g $, $  \alpha f $, $  fg  $ and $  f/g  $
are measurable functions of $  \omega  $ and differentiable ones on
$  \Theta  $ w.p.~1. Since for all of these functions the proofs of
inequality (\ref{equ9}) are quite similar, we verify it for only one of them.
For instance, we examine $  h = fg $.

For all $  \theta_{1}, \theta_{2} \in \Theta  $ we have
\begin{eqnarray*}
\lefteqn{|h(\theta_{1},\omega)-h(\theta_{2},\omega)|} \\
& = & |f(\theta_{1},\omega)g(\theta_{1},\omega)
                                - f(\theta_{2},\omega)g(\theta_{2},\omega)| \\
& = & |f(\theta_{1},\omega)g(\theta_{1},\omega)
  - f(\theta_{2},\omega)g(\theta_{1},\omega)
  + f(\theta_{2},\omega)g(\theta_{1},\omega)
  - f(\theta_{2},\omega)g(\theta_{2},\omega)| \\
& \leq &
  |g(\theta_{1},\omega)||f(\theta_{1},\omega) - f(\theta_{2},\omega)|
+ |f(\theta_{2},\omega)||g(\theta_{1},\omega) - g(\theta_{2},\omega)| \\
& \leq &
  (\lambda_{1}(\omega)\mu_{2}(\omega) + \lambda_{2}(\omega)\mu_{1}(\omega))
  \|\theta_{1} - \theta_{2}\| \quad \mbox{w.p.~1.}
\end{eqnarray*}
In short,
$$
|h(\theta_{1},\omega) - h(\theta_{2},\omega)| \leq
\lambda(\omega) \|\theta_{1} - \theta_{2}\| \quad \mbox{w.p.~1},
$$
where
$$
\lambda = \lambda_{1}\mu_{2} + \lambda_{2}\mu_{1}, \quad
E\lambda = E[\lambda_{1}\mu_{2} + \lambda_{2}\mu_{1}] < \infty.
$$
By Theorem~\ref{the4}, we conclude $  fg \in {\cal D}$.
\end{proof}

Notice, from Lemma~\ref{lem5} (i) and (ii) it follows that being closed for
the operations of addition and multiplication by bounded random variables,
$  {\cal D}  $ is a linear space of functions with these two operations.

\begin{lemma}\label{lem6}
Let $  f,g \in {\cal D} $. Suppose that for any $  \theta_{0} \in \theta $,
there exists a neighbourhood $  {\bf U}_{\omega}(\theta_{0})  $ of
$  \theta_{0}  $ w.p.~1 such that one and only one of the following
conditions
\begin{list}{}{}
\item[{\rm (i)}] $  f(\theta,\omega) = g(\theta,\omega) $,
\item[{\rm (ii)}] $  f(\theta,\omega) < g(\theta,\omega) $,
\item[{\rm (iii)}] $  f(\theta,\omega) > g(\theta,\omega) $
\end{list}
is satisfied for all $  \theta \in {\bf U}_{\omega}(\theta_{0}) $.

Then $  f \vee g \in {\cal D}  $ and $  f\wedge g \in {\cal D} $.
\end{lemma}

\begin{proof}
Consider $  h(\theta,\omega) = f(\theta,\omega) \vee g(\theta,\omega) $. It is clear
that $  h  $ is measurable with respect to $  \omega $. In order to prove
differentiability of $  h  $ w.p.~1 on $  \Theta $, we examine an arbitrary
$  \theta \in \Theta $. There are only two possibility for $  h  $ not to
be differentiable. Firstly, it is possible that the derivative of $  h  $ at
$  \theta  $ does not exist if at least one of the derivatives
$$
\left.\frac{\partial f(\theta,\omega)}{\partial\theta}
\right|_{\theta = \theta_{0}} \quad \mbox{and} \quad
\left.\frac{\partial g(\theta,\omega)}{\partial\theta}
\right|_{\theta = \theta_{0}}
$$
does not. In addition, $  h  $ may not be differentiable at $  \theta  $
if the maximum of the functions $  f  $ and $  g  $ changes over from
$  f  $ to $  g  $ at this point or vice versa. The last case is
equivalent to that there exists $  \omega \in \Omega  $ such that all the
neighborhoods $ {\bf U}_{\omega} (\theta_{0}) \subset \Theta  $ contain
both points at which
$$
f(\theta,\omega) = g(\theta,\omega) \quad \mbox{and} \quad
f(\theta,\omega) \neq g(\theta,\omega).
$$
By the assumption of the lemma, both of these cases may occur only with zero
probability. Therefore, there exists
$  \partial h(\theta,\omega)/\partial\theta |_{\theta = \theta_{0}}  $ at
all $  \theta \in \Theta  $ w.p.~1.

For the function $  h $, the proof will be completed if we show that
$  h  $ satisfies condition (ii) of Theorem~\ref{the4}. Since
$  f,g \in {\cal D} $, there are random variables $  \lambda_{1}  $ and
$  \lambda_{2}  $ with $  E \lambda_{1} < \infty  $ and
$  E\lambda_{2} < \infty  $ such that the inequalities
\begin{eqnarray*}
|f(\theta_{1},\omega) - f(\theta_{2},\omega)| & \leq &
\lambda_{1}(\omega) \|\theta_{1} - \theta_{2}\| \quad \mbox{w.p.~1} \\
|g(\theta_{1},\omega) - g(\theta_{2},\omega)| & \leq &
\lambda_{2}(\omega) \|\theta_{1} - \theta_{2}\| \quad \mbox{w.p.~1}
\end{eqnarray*}
hold for all $  \theta_{1}, \theta_{2} \in \Theta $. Let $  \omega  $ be an
arbitrary element of $  \Omega  $ at which both these inequalities hold.
Divide $  \Theta  $ into two subsets:
\begin{eqnarray*}
{\bf X}_{\omega} & = & \{\theta \in \Theta | f(\theta,\omega)
                                                   \geq g(\theta,\omega)\}, \\
{\bf Y}_{\omega} & = & \{\theta \in \Theta | f(\theta,\omega)
                                                      < g(\theta,\omega)\}.
\end{eqnarray*}
Obviously, it holds
$$
|h(\theta_{1},\omega) - h(\theta_{2},\omega)| \leq \lambda_{1}(\omega)
\| \theta_{1} - \theta_{2} \|
$$
for all $  \theta_{1}, \theta_{2} \in {\bf X}_{\omega}  $ and
$$
|h(\theta_{1},\omega) - h(\theta_{2},\omega)| \leq \lambda_{2}(\omega)
\| \theta_{1} - \theta_{2} \|
$$
for all $  \theta_{1}, \theta_{2} \in {\bf Y}_{\omega} $. Assume
$  \theta_{1} \in {\bf X}_{\omega}, \theta_{2} \in {\bf Y}_{\omega} $. If
$  h(\theta_{1},\omega) \geq h(\theta_{2},\omega) $, we deduce
\begin{eqnarray*}
\lefteqn{|h(\theta_{1},\omega) - h(\theta_{2},\omega)|} \\
& = & |f(\theta_{1},\omega) - g(\theta_{2},\omega)| \\
& < & |f(\theta_{1},\omega) - f(\theta_{2},\omega)|
\leq \lambda_{1}(\omega) \| \theta_{1} - \theta_{2} \|.
\end{eqnarray*}
Similarly, if $  h(\theta_{1},\omega) < h(\theta_{2},\omega)$, we have
$$
|h(\theta_{1},\omega) - h(\theta_{2},\omega)| \leq \lambda_{2}(\omega)
\| \theta_{1} - \theta_{2} \|.
$$
It follows that
$$
|h(\theta_{1},\omega) - h(\theta_{2},\omega)| \leq \lambda(\omega)
\|\theta_{1} - \theta_{2}\|, \quad
\lambda(\omega) = \lambda_{1}(\omega) \vee \lambda_{2}(\omega),
$$
for all $  \theta_{1}, \theta_{2} \in \Theta $. Since this inequality holds
for almost all $  \omega \in \Omega $, we conclude that
$$
|h(\theta_{1},\omega) - h(\theta_{2},\omega)| \leq
\lambda(\omega) \| \theta_{1} - \theta_{2} \| \quad \mbox{w.p.~1},
$$
and
$  E\lambda = E[\lambda_{1} \vee \lambda_{2}]
\leq E\lambda_{1} + E\lambda_{2} < \infty $.

In other words, $  h  $ satisfies the conditions of Theorem~\ref{the4}.
Consequently, $  f \vee g \in {\cal D} $. The proof of the statement
$  f \wedge g \in {\cal D} $, is analogous.
\end{proof}

It should be noted that the condition of Lemma~\ref{lem6} is not necessary, as
the next example shows.
\begin{example}\label{exa2}
{\rm
Let $  \Theta = [-1,1] $, $  (\Omega, {\cal F}, P)  $ be a probability
space, where $  \Omega = [0,1] $, $  {\cal F}  $ is the $\sigma$--field of
Borel sets of $  \Omega $, and $  P  $ is the Lebesgue measure on
$  \Omega $. Consider the following functions:
$$
f(\theta,\omega) = -\theta^{3} + \omega, \qquad
g(\theta,\omega) = \theta^{2} + \omega
$$
and
$$
h(\theta,\omega) = f(\theta,\omega) \vee g(\theta,\omega)
= \left\{
    \begin{array}{ll}
     -\theta^{3} + \omega, & \mbox{if $  -1 \leq \theta \leq 0 $}, \\
      \theta^{2} + \omega, & \mbox{if $  0 < \theta \leq 1 $}.
    \end{array}
  \right.
$$

One can easily verify that for any neighbourhood of $  \theta = 0 $, there
exist both points with $  f > g  $ and $  f < g  $ w.p.~1. The conditions
of Lemma~\ref{lem6} are therefore violated. Nevertheless, $  h  $ is
differentiable at $  0  $ for all $  \omega \in \Omega $. Moreover, it
holds that $  h \in {\cal D} $.
}
\end{example}

\begin{corollary}\label{cor7}
Let $  f,g \in {\cal D} $. If for every $  \theta \in \Theta  $ it holds
that $  f \neq g  $ w.p.~1, then $  f \vee g \in {\cal D}  $ and
$  f \wedge g \in {\cal D} $.
\end{corollary}

\begin{proof}
Clearly, the condition of the corollary implies that either
$  f-g > 0  $ or $  f-g < 0  $ holds at every $  \theta \in \Theta  $
w.p.~1. Since $  f,g \in {\cal D} $, these two functions are continuous
functions of $  \theta  $ w.p.~1 as well as $  f-g $. Because of
continuity, $  f-g > 0 \; (f-g < 0)  $ holds w.p.~1 not only at
$  \theta $, but also at every points of a neighbourhood of $  \theta $. It
remains to apply Lemma~\ref{lem6}.
\end{proof}

Using Corollary~\ref{cor7}, we give the following general conditions for
$  {\cal D}  $ to provide closeness with respect to the operations
$  \vee  $ and $  \wedge $.

\begin{lemma}\label{lem8}
Let $  f,g \in {\cal D}  $. If for any $  \theta \in \Theta  $ it holds
that the random variables $  f(\theta,\omega)  $ and $  g(\theta,\omega) $
\begin{list}{}{}
\item[{\rm (i)}] are independent,
\item[{\rm (ii)}] at least one of them is continuous,
\end{list}
then $  f \vee g \in {\cal D}  $ and $  f \wedge g \in {\cal D} $.
\end{lemma}

To prove the lemma it is sufficient to see that its conditions lead to that of
Corollary~\ref{cor7}.

The next two examples show that both conditions of Lemma~\ref{lem8} are
essential.
\begin{example}\label{exa3}
{\rm
Let $  (\Omega, {\cal F}, P)  $ and $  \Theta  $ be defined as in
Example~\ref{exa2}. Also define
$$
f(\theta,\omega) = -\theta+\omega \quad \mbox{and} \quad
g(\theta,\omega) = \theta+\omega.
$$
Let us consider the function
$$
h(\theta,\omega) = f(\theta,\omega) \vee g(\theta,\omega)
= \left\{
   \begin{array}{ll}
    -\theta + \omega, & \mbox{if $  -1 \leq \theta \leq 0 $}, \\
     \theta + \omega, & \mbox{if $  0 < \theta \leq 1 $}.
   \end{array}
  \right.
$$
It is clear that $  f,g \in {\cal D}  $ and for every
$  \theta \in \Theta $, the random variables $  f(\theta,\omega)  $ and
$  g(\theta,\omega)  $ are continuous. Although inequality (\ref{equ9})
holds with $  \lambda = 1  $ for $  h $, this function is not
differentiable at $  \theta = 0  $ for all $  \omega \in \Omega $.
Therefore, $  h \not\in {\cal D} $.
}
\end{example}

\begin{example}\label{exa4}
{\rm
Let $  \Theta = [0,1] $, $  \Omega_{1} = \Omega_{2} = [0,1]  $ and
$  P  $ be the Lebesgue measure on
$  \Omega = \Omega_{1} \times \Omega_{2} $. Denote
$  \omega = (\omega_{1},\omega_{2})^{\top}  $ and consider the following
functions:
$$
f(\theta,\omega)
=
  \left\{
    \begin{array}{ll}
      \frac{1}{2} \theta, & \mbox{if $  \omega_{1} \leq \frac{1}{2} $}, \\
                 1, &  \mbox{if $  \omega_{1} > \frac{1}{2} $},
    \end{array}
  \right.
\qquad
g(\theta,\omega)
=
  \left\{
    \begin{array}{ll}
      \theta^{2}, & \mbox{if $  \omega_{2} \leq \frac{1}{2} $}, \\
       1, & \mbox{if $  \omega_{2} > \frac{1}{2} $},
\end{array}
\right.
$$
and
$$
h(\theta,\omega) = f(\theta,\omega) \vee g(\theta,\omega)
= \left\{
    \begin{array}{ll}
      \max \{\frac{1}{2} \theta, \theta^{2}\}, &
                              \mbox{if $  \omega_{1} \leq \frac{1}{2}  $ and
                                       $  \omega_{2} \leq \frac{1}{2} $}, \\
      1, & \mbox{otherwise}.
    \end{array}
  \right.
$$
One can see that $  f,g \in {\cal D}  $ and for every
$  \theta \in \Theta $, the random variables $  f(\theta,\omega)  $ and
$  g(\theta,\omega)  $ are independent. In addition, condition (ii) of
Theorem~\ref{the4} holds for $  h  $ with $  \lambda = 2 $. Nevertheless,
$  h = \max \{\frac{1}{2}\theta,\theta^{2}\}  $ with probability
$  \frac{1}{4} $, that is not a differentiable function at
$  \theta = \frac{1}{2} $. In that case, $  h \not\in {\cal D} $.
}
\end{example}

\begin{lemma}\label{lem9}
Let $  {\cal M}  $ be a set of functions from $  {\cal D}  $ such that for
any $  f,g \in {\cal M} $, the conditions of Lemma~\ref{lem6} are fulfilled.
Then $   {\cal M}  $ is closed for the operations $  \max  $ and
$  \min $.
\end{lemma}

\begin{proof}
Let $  f,g \in {\cal M}  $ and let us define
$  h =  f \vee g $. Note that $  h \in {\cal D}  $ by Lemma~\ref{lem6}. We
have to prove the conditions of Lemma~\ref{lem6} are satisfied for $  h  $
and any $  u \in {\cal M} $.

If $  u  $ is either $  f  $ or $  g $, say $  u \equiv f $, we may
write
$$
h-u = f \vee g - f
= \left\{
    \begin{array}{ll}
       g - f, & \mbox{if $  f < g $}, \\
       0, & \mbox{if $  f \geq g $}.
    \end{array}
  \right.
$$
Since $  f,g \in {\cal M} $, for any point of $  \Theta $, there is a
neighbourhood on which only one of the conditions $  f-g < 0 $,
$  f-g = 0 $, or $  f-g > 0  $ holds w.p.~1. From the above identity this
also holds for $  h-u  $ on the neighbourhood. Consequently, in this case
the conditions of Lemma~\ref{lem6} are fulfilled.

Now we assume $  u \in {\cal M} \setminus \{f,g\} $. We have
$$
h-u = f \vee g - u
= \left\{
    \begin{array}{ll}
      g - f, & \mbox{if $  f < g $}, \\
      f - u, & \mbox{if $  f \geq g $}.
    \end{array}
  \right.
$$

Let us examine any $  \theta \in \Theta $. Suppose that $  f < g  $ w.p.~1
at $  \theta $. Since $  f,g $, and $  u  $ belong to $  {\cal M} $,
there are neighborhoods $  {\bf U}_{\omega}(\theta)  $ and
$  {\bf V}_{\omega}(\theta)  $ where the conditions of Lemma~\ref{lem6} are
fulfilled for each pairs of functions $  (f,g)  $ and $  (g,u) $,
respectively. It follows from the above expression that the neighborhood
$  {\bf U}_{\omega} \bigcap {\bf V}_{\omega}(\theta)  $ is that
Lemma~\ref{lem6} requires for $  h  $ and $  u $. If it holds that
$  f \geq g  $ or $  f=g  $ at $  \theta $, the reasoning is the same.

In short, we have shown that the conditions of Lemma~\ref{lem6} are satisfied
for $  h  $ and any $  u \in {\cal M}  $ and, therefore,
$  h = f \vee g \in {\cal M} $. In the case of minimum, the proof is
analogous.
\end{proof}

\begin{corollary}\label{cor10}
If $  f_{j} \in {\cal M}  $ for every $  j = 1,\ldots,N $, then it holds
$$
\bigvee_{i \in I} \bigwedge_{j \in J_{i}} f_{j} \in {\cal M},
$$
where $  I  $ is a finite set of indices and
$  J_{i} \subset \{1,\ldots,N\}  $ for every $  i \in I $.
\end{corollary}

This is an immediate consequence of the previous lemma.

The next example is of importance to the main result of the section.

\begin{example}\label{exa5}
{\rm
Let $  f_{j} \in {\cal D}  $ for all $  j=1,\ldots,N $. Suppose that at
every $  \theta \in \Theta $, all the random variables
$  f_{j}(\theta,\omega)  $ are continuous and independent. Define
$  {\cal L}  $ to be a set of linear combinations
$  \sum_{i \in I} a_{i} f_{i}  $ with integer coefficients $  a_{i} $,
$  i \in I \subset \{1,\ldots,N\} $. Obviously, $  {\cal L}  $ is stable
for addition. For all functions $  u = \sum_{i \in I} a_{i} f_{i}  $ and
$  v = \sum_{j \in J} b_{j} f_{j} $, we have
$  u-v = \sum_{k \in K} c_{k} f_{k} $. It is clear that for every
$  \theta \in \Theta $, $  u-v  $ is a continuous random variable because
of the properties of $  f  $ (except for the case of all $  c_{k} = 0  $
which is obvious). Therefore, it holds that $  u-v \neq 0  $ w.p.~1 at every
$  \theta \in \Theta $. Similarly as in Corollary~\ref{cor7}, one can deduce
that $  u  $ and $  v  $ satisfy the conditions of Lemma~\ref{lem6}. From
this we conclude that $  {\cal L}  $ may be treated as an example of
$  {\cal M} $.
}
\end{example}

One can easily see that the condition of continuity is essential to this
reasoning. To illustrate the important role of independence, consider the
following functions
$$
f(\theta,\omega) = -2\theta+2\omega, \quad
g(\theta,\omega) = \theta-\omega, \quad \mbox{and} \quad
u(\theta,\omega) = \theta+\omega,
$$
under the same assumption as in Example~\ref{exa3}. It is easy to verify that
the conditions of Lemma~\ref{lem6} are fulfilled for any two functions of
them. Nevertheless, the functions $  u  $ and $  v=f+g  $ do not satisfy
the conditions, as Example~\ref{exa3} has shown.

Now, we may formulate the main result of the section. We first introduce some
definitions. Let $  {\cal A}  $ be the algebra of all functions
$  f: \Theta \times \Omega \longrightarrow R  $ being defined on the
probability space $  (\Omega, {\cal F}, P)  $ at every
$  \theta \in \Theta  $ with the operations $  \vee $, $  \wedge $, and
$  + $. In other words, this is a closed system of the functions for these
operations.

\begin{definition}
Let $  {\bf T}  $ be a finite subset of functions of $  {\cal A} $. We
define $  [{\bf T}]_{\cal A}  $ to be the set generated by $  {\bf T}  $
in $  {\cal A} $, that is the set of all functions being obtained from ones
of $  {\bf T}  $ by means of the operations $  \vee $, $  \wedge $, and
$  + $.
\end{definition}

\begin{theorem}\label{the11}
Let $  {\bf T} \in {\cal D} $. Suppose that for all $  \tau \in {\bf T} $,
$  \tau(\theta,\omega)  $ are continuous and independent random variables at
any $  \theta \in \Theta $.

Then it holds $  [{\bf T}]_{\cal A} \subset {\cal D} $.
\end{theorem}

\begin{proof}
It results from Lemma~\ref{lem3} that every
$  f \in [{\bf T}]_{\cal A}  $ can be represented as
$$
f = \bigvee_{i \in I} \bigwedge_{j \in J_{i}} \sum_{\tau \in {\bf T}}
a_{ij}^{\tau} \tau,
$$
where all $  a_{ij}^{\tau}  $ are integers. It has been shown in
Example~\ref{exa5} that the functions of the family
$  \{ \sum_{\tau \in {\bf T}} a_{k}^{\tau} \tau \}_{k=1,2,\ldots}  $ satisfy
the conditions of Lemma~\ref{lem6}. Applying Corollary~\ref{cor10}, we
conclude that the statement of the theorem is true.
\end{proof}

It is important to note that the conditions of Theorem~\ref{the11} are rather
general and usually fulfilled in the network simulation. In particular, in
contrast with the traditional approaches (cf., for example, existing results
on the unbiasedness of IPA estimates in \cite{HoYC87,Suri89}), we may not
restrict ourselves to the exponential distribution.

In short, to satisfy the theorem only the following are required for the
functions of the set $  {\bf T} $:
\begin{list}{}{}
\item[(i)] for any $  \theta \in \Theta $, all $  \tau \in {\bf T}  $ are
continuous and independent random variables;
\item[(ii)] each $  \tau \in {\bf T}  $ as a function of $  \theta $ is
differentiable w.p.~1 and Lipschitz one with an integrable random variable as
a Lipschitz constant.
\end{list}

In the next section we will show how these results can be applied to some
problems to verify the unbiasedness of gradient estimates.

\section{Applications}\label{sec5}

Now we discuss the applications of the previous results to optimizing the
networks. In particular, we describe algorithms of obtaining sample gradients,
based on the algebraic representation of the networks. In this section we keep
using the notations $  (\Omega, {\cal F}, P)  $ and $  \Theta  $ for the
underlying probability space and the parameter space, respectively.

We begin with the stochastic activity network. Let the duration of the $j$th
activity be represented by the function $  \tau_{j} (\theta,\omega) $. Denote
the set of all such functions of the network by $  {\bf T} $. As we have
seen, a sample completion time of the network $  t(\theta,\omega)  $ may be
expressed by functions of $  {\bf T}  $ by using only the operations
$  \max  $ and $  + $. This implies $  t \in [{\bf T}]_{\cal A} $.

Suppose that $  {\bf T} \in {\cal D} $, and all $  \tau \in {\bf T}  $ are
continuous and independent random variables at every $  \theta \in \Theta $.
For the mean completion time $  T(\theta) = E[t(\theta,\omega)] $, it follows
from Theorem~\ref{the11} that
$  (1/N) \sum_{i=1}^{N} \partial t(\theta,\omega_{i})/\partial\theta $, where
$  \omega_{i} \in \Omega $, is an unbiased estimate of the gradient
$  \partial T(\theta)/\partial\theta $.

As an example, suppose $  \tau(\theta,\omega) = -\theta\ln(1-\omega) $, where
$  \theta \in R  $ and the random variable $  \omega  $ is uniformly
distributed on $  [0,1] $. It is well known \cite{Erma75} that
$  -\ln(1-\omega)  $ has an exponential distribution with mean 1. Similarly
as in Example~\ref{exa1}, we have $  \tau \in {\cal D} $. In addition,
durations of the activities are normally considered as independent in the
probabilistic sense. Our results are therefore applicable in this case.

Now suppose that there is a simulation procedure for the activity network with
$  L  $ nodes to provide a simulation experiment for any fixed
$  \theta \in \Theta  $ and a realization of $  \omega $. One can easily
combine it with the following algorithm.

\paragraph*{Algorithm 1}
\nopagebreak[4]
\begin{list}{}{}
\item[Step (i).] At the initial time, fix values of $  \theta  $ and
$  \omega $; set $  g_{j} = 0  $ for $  j=1,\ldots,L $, and set
$  c = 0 $.
\item[Step (ii).] Upon the completion of any activity $  i $, add the value
of $  \partial \tau_{i}(\theta,\omega)/\partial\theta  $ to $  g_{i}  $
and add $  1  $ to $  c $; if $  c = L $, then save $  g_{i}  $ as the
value of $  \partial t(\theta,\omega)/\partial\theta  $ and stop; otherwise
go to Step (iii).
\item[Step (iii).] Determine the set $ {\bf N}_{D}(i) $. For every
$  j \in {\bf N}_{D}(i) $, if all activities of the set
$  {\bf N}_{F}(j)  $ have been completed, then set $  g_{j} = g_{i} $.
\end{list}

To verify the correctness of Algorithm~1, it suffices to see that it is
simply based on recursive equation (\ref{equ1}).

For a reliability network, one can apply Theorem~\ref{the11} in a similar way.
As in Section~\ref{sec2}, denote the sample lifetime of a system by
$  t(\theta, \omega)  $. It is not difficult to construct the next algorithm
that calculates the sample gradient
$  \partial t(\theta,\omega)/\partial\theta $.

\paragraph*{Algorithm 2}
\begin{list}{}{}
\item[Step (i).] At the initial time, fix values of $  \theta  $ and
$  \omega $.
\item[Step (ii).] Upon the failure of element $  i $, exclude all nodes
representing the elements that are now not able to keep working from the set
$  {\bf N}  $ as well as the corresponding arcs from the set $  {\bf A} $.
\item[Step (iii).] If for the reduced set $  {\bf N}  $ it holds
$  {\bf N} \cap {\bf N}_{E} = \emptyset $, then save
$  \partial \tau_{i}(\theta,\omega)/\partial\theta  $ as the value of
$  \partial t(\theta,\omega)/\partial\theta  $ and stop; otherwise go to
Step (ii).
\end{list}

Finally, we consider the queueing network which is a rather complicated model.
Applying Theorem~\ref{the11} to a network with a deterministic routing
mechanism, we may conclude that
\begin{list}{}{}
\item[(i)] if $  {\bf T} \in {\cal D} $, and for all $  \tau \in {\bf T} $,
$  \tau(\theta,\omega)  $ are continuous and independent random variables
for every $  \theta \in \Theta $, then the estimate (\ref{equ7}) is unbiased
for both the expected average total time $  T  $ and the expected average
waiting time $  W $;
\item[(ii)] if in addition to previous assumptions, for all
$  \tau_{1},\tau_{2} \in T $, condition (iv) of Lemma~\ref{lem5} is
fulfilled, then the estimate (\ref{equ7}) is unbiased for the expected average
utilization $  U $, the expected average number of customers $  C  $ and
the expected average queue length $  Q $.
\end{list}

In the case of the stochastic routing mechanism with a random routing table,
the above conclusions still hold true. This follows from representation
(\ref{equ4}) and Lemma~\ref{lem5} (ii) because of the boundedness of the
indicator random variable.

In order to construct a useful algorithm of calculating a sample performance
function gradient, we first consider the identities at (\ref{equ3}) in the
form
$$
\delta_{ij}(\theta,\omega)
= \max \{\alpha_{ij}(\theta,\omega),\delta_{ij-1}(\theta,\omega) \}
+ \tau_{ij}(\theta,\omega).
$$
Differentiating the function $  \delta_{ij}(\theta,\omega)  $ for a fixed
$  \omega \in \Omega  $ at $  \theta $, we have
$$
\frac{\partial}{\partial \theta} \delta_{ij}(\theta,\omega)
= \left\{
   \begin{array}{ll}
    \frac{\partial}{\partial\theta} \delta_{ij-1} (\theta,\omega)
  + \frac{\partial}{\partial\theta} \tau_{ij} (\theta,\omega), &
\mbox{if $  \alpha_{ij}(\theta,\omega) < \delta_{ij-1}(\theta,\omega) $}, \\
    \frac{\partial}{\partial\theta} \alpha_{ij} (\theta,\omega)
  + \frac{\partial}{\partial\theta} \tau_{ij} (\theta,\omega), &
\mbox{if $  \alpha_{ij}(\theta,\omega) > \delta_{ij-1}(\theta,\omega) $}.
   \end{array}
  \right.
$$
The inequalities in the conditions of the right-hand side mean the following.
The inequality $  \alpha_{ij} > \delta_{ij-1}  $ implies that at the arrival
of the $j$th customer into node $i$, the service of the previous one, the
$(j-1)$th, has been completed. Therefore, at that time the server is free. The
meaning of the contrary inequality is that the server is busy at the arrival
of the $j$th customer.

In the first case, the value of
$  \partial \delta_{ij}(\theta,\omega)/\partial\theta $ is defined by
$  \partial \alpha_{ij}(\theta,\omega)/\partial\theta $. Note that the
function $  \alpha_{ij}(\theta,\omega) \equiv \delta_{km}(\theta,\omega) $,
where $  \delta_{km}(\theta,\omega)  $ represents some $m$th completion time
at a node $ k $. It is the service completion of the customer that is the
$j$th to arrive into node $i$. In other words, in that case one have to add
the value of $  \partial \delta_{km} (\theta,\omega)/\partial\theta  $ to
$  \partial \tau_{ij}(\theta,\omega)/\partial\theta  $ to obtain
$  \partial \delta_{ij}(\theta,\omega)/\partial\theta $.

If it holds $  \alpha_{ij} < \delta_{ij-1} $, then the value of
$  \partial \delta_{ij}(\theta,\omega)/\partial\theta  $ is calculated by
addition $  \partial \delta_{ij-1}(\theta,\omega)/\partial\theta  $ and
$  \partial \tau_{ij}(\theta,\omega)/\partial\theta $.

It results from Section~\ref{sec4} that if for any $  \theta \in \Theta $,
all $  \tau \in {\bf T}  $ are continuous and independent random variables,
then $  \alpha_{ij}(\theta,\omega) \neq \delta_{ij-1}(\theta,\omega)  $
w.p.~1. Therefore, the above expression defines the sample gradient w.p.~1.

Now, we consider an algorithm that provide the value of
$  \partial \delta_{KM}(\theta,\omega)/\partial\theta  $ for fixed
$  \theta \in \Theta \subset R $, $  \omega \in \Omega $,
$  K \in \{1, \ldots, L\} $, and $  M \in \{1, 2, \ldots \} $. We suppose
that there is a simulation procedure into which the algorithm may be
incorporated.

\paragraph*{Algorithm 3}
\begin{list}{}{}
\item[Step (i).] At the initial time, fix values of $  \theta  $ and
$  \omega $; set $  g_{j} = 0  $ for $  j=1,\ldots,L $.
\item[Step (ii).] Upon the $j$th completion at node $  i $, add the value of
$  \partial \tau_{ij}(\theta,\omega)/\partial\theta  $ to  $  g_{i} $; if
both $  i = K $, and $  j = M $, then save $  g_{K}  $ as the value of
$  \partial \delta_{KM}(\theta,\omega)/\partial\theta  $ and stop;
otherwise go to Step (iii).
\item[Step (iii).] Determine the next node $  r = \sigma_{ij}(\omega)  $ to
be visited by the customer; if the server of node $  r  $ is free, then set
$  g_{r} = g_{i} $.
\end{list}

Note that in the case of a vector of parameters, $  \Theta \subset R^{n} $,
the algorithm is analogous. It only needs to change $  g_{i}  $ for the
vector $  \mbox{\boldmath $g$}_{i} = (g_{i1},\ldots,g_{in} )^{\top}  $ and
to treat the arithmetic operations as the vector ones.

It is easy to see that the algorithm of evaluating
$  \delta_{ij}(\theta,\omega)/\partial\theta  $ plays the key role in
calculating gradients of sample performance functions of the network. It is
included as the main part in other algorithms. To illustrate this, we consider
an algorithm for the sample function
$  u(\theta,\omega) = \sum_{j=1}^{M} \tau_{Kj}(\theta,\omega)
/\delta_{KM}(\theta,\omega) $, the average utilization per unit time.

\paragraph*{Algorithm 4}
\begin{list}{}{}
\item[Step (i).] At the initial time, fix values of $  \theta  $ and
$  \omega $; set $  t,d = 0 $, and $  g_{j} = 0  $ for $  j=1,\ldots,L $.
\item[Step (ii).] Upon the $j$th completion at node $  i $, add the value of
$  \partial \tau_{ij}(\theta,\omega)/\partial\theta  $ to  $  g_{i} $; if
$  i = K $, then add the value of
$  \partial \tau_{ij}(\theta,\omega)/\partial\theta  $ to $  d $, and add
the value of $  \tau_{Kj}(\theta,\omega)  $ to $  t $; if both $  i = K $,
and $  j = M $, then set $  h = \delta_{KM}(\theta,\omega)  $ and stop;
otherwise go to Step (iii).
\item[Step (iii).] Determine the next node $  r = \sigma_{ij}(\omega)  $ to
be visited by the customer; if the server at node $  r  $ is free, then set
$  g_{r} = g_{i} $.
\end{list}

Upon the completion of the algorithm we get $  (dh - tg_{K})/h^{2}  $ as the
value of $  \partial u(\theta,\omega)/\partial\theta $.

It is easy to see that Algorithms~3 and 4 are quite similar to those of
IPA method in \cite{Suri89}.

In conclusion, note that the algorithms are rather simple. In fact, they only
require calculating gradients of given functions and performing some trivial
operations to produce values of the sample gradients. Using these values, one
can easily estimate the gradients of network performance measures so as to
apply efficient optimization procedures.

\subsection*{Acknowledgements}
The author is grateful to Prof. A.A.~Zhigljavsky who was his PhD supervisor
and Prof. S.M.~Ermakov for help and valuable discussions.

\bibliography{Unbiased_Estimates_for_Gradients_of_Stochastic_Network_Performance_Measures}

\end{document}